\documentclass{amsart}

\usepackage[bookmarksnumbered,colorlinks, plainpages]{hyperref}
\hypersetup{colorlinks=true,linkcolor=red, anchorcolor=green, citecolor=cyan, urlcolor=red, filecolor=magenta, pdftoolbar=true}

\usepackage{amsmath, amsthm, amscd, amsfonts, amssymb, graphicx, color, enumerate, xcolor, mathrsfs, latexsym}

\usepackage{tikz}
\usetikzlibrary{arrows}

\theoremstyle{plain}
\newtheorem{theorem}{Theorem}[section]
\newtheorem{lemma}[theorem]{Lemma}

\newtheorem{corollary}[theorem]{Corollary}

\theoremstyle{definition}
\newtheorem*{definition}{Definition}

\newcommand{\B}{\mathcal{B}}

\newcommand{\G}{\Gamma}

\newcommand{\parn}[1]{\left(#1\right)}
\newcommand{\1}{\mathbf{1}}

\newcommand{\J}{\mathcal{J}}
\newcommand{\tr}{\mathrm{tr}}

\begin{document}
\title[Ladder, circular ladder, and M\"{o}bius graphs]{Explicit formulas for matrices associated to ladder, circular ladder, and M\"{o}bius ladder graphs}
\author{Ali Azimi}
\address{Department of Mathematics and Applied Mathematics, Xiamen University Malaysia, 43900, Sepang, Selangor Darul Ehsan, Malaysia}
\email{ali.azimi61@gmail.com,\ ali.azimi@xmu.edu.my}

\author{M. Farrokhi D. G.}
\address{Research Center for Basic Sciences and Modern Technologies (RBST), Institute for Advanced Studies in Basic Sciences (IASBS), Zanjan 45137-66731, Iran}
\email{m.farrokhi.d.g@gmail.com,\ farrokhi@iasbs.ac.ir}


\begin{abstract}
We give explicit formulas for resistance distance matrices and Moore-Penrose inverses of incidence and Laplacian matrices of ladder, circular ladder, and M\"{o}bius ladder graphs. As a result, we compute the Kirchhoff index of these graphs and give new combinatorial formulas for the number of their spanning trees.
\end{abstract}
\maketitle
\section{introduction}
The incidence matrix $Q(\Gamma)$ of an oriented graph $\Gamma$ with vertex-set $V(\Gamma) = \{v_1,\ldots, v_n\}$ and edge-set $E(\Gamma) = \{e_1,\ldots, e_m\}$ is an  $n\times m$ matrix defined as follows: The rows and columns of $Q$ are indexed by vertices and edges of $\Gamma$, respectively. The $(i, j)$ entry of $Q$ is $0$ if the vertex $v_i$ and the edge $e_j$ are not incident, otherwise it is $1$ or $-1$ depending on whether $e_j$ starts from or ends at $v_i$. The Laplacian matrix $L(\Gamma)$ of $\Gamma$, that is actually equal to $Q(\Gamma)Q'(\Gamma)$, is an $n\times n$ matrix whose rows and columns are indexed by vertices of $\Gamma$. The $(i, j)$ entry of $L(\Gamma)$ is equal to the degree of the vertex $v_i$ if $i = j$, and it is $-1$ or $0$ if the vertices $v_i$ and $v_j$ are adjacent or non-adjacent, respectively. 

The Moore-Penrose inverse of an $m\times n$ real matrix $A$, denoted by $A^+$, is the (unique) $n\times m$ real matrix that satisfies the following equations 
\[AA^+A=A,\hspace{0.4cm} A^+AA^+=A^+,\hspace{0.4cm} (AA^+)'=AA^+,\hspace{0.4cm} (A^+A)'=A^+A\]
(see \cite{ab-tneg, slc-cdm}). The Moore-Penrose inverse of the incidence matrix of a graph was first studied by Ijira in \cite{yi}. We refer the interested reader to \cite{aa-rbb:2018, aa-rbb:2019, aa-rbb-ee} for previous studies on the Moore-Penrose inverses of matrices associated to various classes of graphs.

Incidence matrices and their inverses, such as the Moore-Penrose inverses, are important tools used in analysis of graphs. Let $\Gamma$ be a simple connected graph. The resistance distance between two vertices $v_i$ and $v_j$ can be computed via the formula
\[r(v_i,v_j)=L^+_{v_i,v_i}+L^+_{v_j,v_j}-L^+_{v_i,v_j}-L^+_{v_j,v_i},\]
where $L^+_{v_i,v_j}$ denotes the $(i, j)$-entry of the Moore-penrose inverse $L^+$ of $L=L(\Gamma)$. The notion of resistance distance between vertices of a graph is introduced by Klein and Randic in 1993 \cite{djk-mr} as the effective resistance between the given vertices (computed by Ohm's law) when a battery is attached between them assuming that all edges are unit resistors.  Beside many applications in physics, chemistry, and network sciences (see \cite{ee-md, mk-bm, yy-sy}), resistance distances can also be studied from a purely mathematical point of view as they represent new metrics on graph as well as making connections to graph theoretical objects like the Moore-Penrose inverse of Laplacian matrices above, see \cite{pa-fa-rbb, rbb-sg, nlb, kcd-yy, xg-yl-wl, svg, ah-hc-qd}. A strictly related parameter is the so-called Kirchhoff index of the graph, namely 
\[Kf(\G) = \frac{1}{2}\sum_{i=1}^{n}\sum_{j=1}^{n}r(v_i,v_j).\]

In this paper, we consider three familiar classes of graphs, namely ladder, circular ladder, and M\"{o}bius graphs, and study their resistance distances and the Moore-Penrose inverse of their Laplacian matrices. Recall that the ladder graph $L_n$ on $2n$ vertices is composed of two paths of length $n-1$ with spokes/rungs connecting corresponding vertices. The circular ladder graph $CL_n$ defines analogous to ladder graphs with paths replaced by cycles. Also, the M\"{o}bius ladder graph $M_n$ obtains from the circular ladder graph by switching two non-spoke edges of a square.
\section{Preliminaries}
In this section, we introduce some recursive sequences arising from the study of ladder, circular ladder, and M\"{o}bius ladder graphs. 
\begin{definition}
By a $(p,q)$-sequence $\{x_n\}$ we mean a sequence satisfying $x_n=px_{n-1}-qx_{n-2}$ for all $n$ provided that $x_{n-1}$ and $x_{n-2}$ are defined.
\begin{itemize}
\item $\{a_n\}_{n=0}^\infty$ is the $(4,1)$-sequence \href{https://oeis.org/A001835}{A001835} with initial values $(a_0,a_1)=(1,1)$.
\item $\{s_n\}_{n=0}^\infty$ is the $(4,1)$-sequence \href{https://oeis.org/A001353}{A001353} with initial values $(s_0,s_1)=(0,1)$.
\end{itemize}
\end{definition}

Lemma \ref{sequences} presents a list of identities satisfied by these sequences and will be used frequently in the rest of the paper.
\begin{lemma}\label{sequences}
The following identities hold among sequences $\{a_n\}$ and $\{s_n\}$:
\begin{align}
\label{s_n=a_1+...+a_n}
s_n&=\sum_{i=1}^na_i,\\
\label{2s_n=a_(n+1)-a_n}
2s_n&=a_{n+1}-a_n,\\
\label{a_(n+1)-a_n=a_(i+1)a_(n-i+1)-a_ia_(n-i)}
a_{n+1}-a_n&=a_{k+1}a_{n-k+1}-a_ka_{n-k},\\
\label{s_n=a_is_(n-i)+a_(n-i+1)s_i}
s_n&=a_ks_{n-k}+a_{n-k+1}s_k,\\
\label{(a_(n+1)+a_n)^2=3(a_(n+1)-a_n)^2+4}
(a_{n+1}+a_n)^2&=3(a_{n+1}-a_n)^2+4,\\
\label{a_1a_n+...+a_na_1=(n+1)s_n+2na_n)/6}
\sum_{i+j=n+1}^na_ia_j&=\frac{n(a_{n+1}+a_n)+(a_{n+1}-a_n)}{6}=\frac{(n+1)s_n+na_n}{3}
\end{align}
for all $n\geq k\geq 0$.
\end{lemma}
\begin{proof}
The equations \eqref{s_n=a_1+...+a_n} and \eqref{2s_n=a_(n+1)-a_n} follow by induction on $n$ and the equation \eqref{a_(n+1)-a_n=a_(i+1)a_(n-i+1)-a_ia_(n-i)} follows by induction on $k$. Let $\lambda:=2+\sqrt{3}$, $\alpha_a:=(3-\sqrt{3})/6$, $\beta_a:=(3+\sqrt{3})/6$, and $\alpha_s=-\beta_s=\sqrt{3}/6$. Then $a_n$ and $s_n$ are given by the Binet formulas 
\[a_n=\alpha_a\lambda^n+\beta_b\lambda^{-n}\quad\text{and}\quad s_n=\alpha_s\lambda^n+\beta_s\lambda^{-n}.\]
The equations \eqref{s_n=a_is_(n-i)+a_(n-i+1)s_i} and \eqref{(a_(n+1)+a_n)^2=3(a_(n+1)-a_n)^2+4} can be obtained simply from the Binets formulas. Now, we prove equation \eqref{a_1a_n+...+a_na_1=(n+1)s_n+2na_n)/6}. Let $b_n:=\sum_{i=1}^na_ia_{n+1-i}$ for all $n\geq1$. First observe that
\begin{align*}
b_n=\sum_{i=1}^na_ia_{n+1-i}&=\sum_{i=1}^n(4a_ia_{n-i}-a_ia_{n-1-i}),
\end{align*}
which implies that 
\begin{equation}\label{b_n: 1}
b_n-4b_{n-1}+b_{n-2}=a_n-a_{n-1}
\end{equation}
On the other hand, by equation \eqref{a_(n+1)-a_n=a_(i+1)a_(n-i+1)-a_ia_(n-i)},
\begin{align*}
b_n-b_{n-2}&=a_1a_n+\parn{(a_2a_{n-1}-a_1a_{n-2})+\cdots+(a_{n-1}a_2-a_{n-2}a_1)}+a_na_1\\
&=2a_n+(n-2)(a_n-a_{n-1})
\end{align*}
so that 
\begin{equation}\label{b_n: 2}
b_n-b_{n-2}=na_n-(n-2)a_{n-1}.
\end{equation}
Summing up the equations \eqref{b_n: 2} when $n$ is replaced by $3,\ldots,n$, it follows that
\begin{equation}\label{b_n: 3}
b_n+b_{n-1}=na_n+s_{n-1}.
\end{equation}
Now, by solving the system of linear equations \eqref{b_n: 1}, \eqref{b_n: 2}, and \eqref{b_n: 3} for $b_n$, $b_{n-1}$, and $b_{n-2}$, the result follows.

\end{proof}	
\section{ladder graph}
Let $L_n$ denote the ladder graph of order $2n$. Suppose vertices along the top and bottom of $L_n$, as depicted in Fig. \ref{ladder graph}, are labeled $u_1^+,\ldots, u_n^+$ and $u_1^-,\ldots, u_n^-$, respectively. Also, suppose the edges of $L_n$ are oriented as $e_i^\varepsilon=(u_i^\varepsilon,u_{i+1}^\varepsilon)$ for $i=1,\ldots,n-1$ with $\varepsilon=\pm1$, and $f_i=(u_i^+,u_i^-)$ for $i=1,\ldots,n$, where $f_i$ are the spokes.
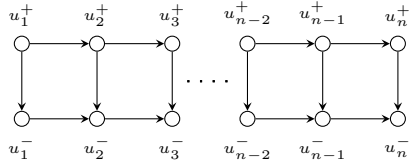
\begin{figure}[h]
\begin{tikzpicture}[>=stealth]
\draw [loosely dotted,thick] (2.2,0.5)--(2.8,0.5);
\node [draw,circle,fill=white,inner sep=2pt,label=below:\tiny{$u_1^-$}] (A) at (0,0) {};
\node [draw,circle,fill=white,inner sep=2pt,label=below:\tiny{$u_2^-$}] (B) at (1,0) {};
\node [draw,circle,fill=white,inner sep=2pt,label=below:\tiny{$u_3^-$}] (C) at (2,0) {};
\node [draw,circle,fill=white,inner sep=2pt,label=below:\tiny{$u_{n-2}^-$}] (D) at (3,0) {};
\node [draw,circle,fill=white,inner sep=2pt,label=below:\tiny{$u_{n-1}^-$}] (E) at (4,0) {};
\node [draw,circle,fill=white,inner sep=2pt,label=below:\tiny{$u_n^-$}] (F) at (5,0) {};

\draw [loosely dotted,thick] (2.2,0.5)--(2.8,0.5);
\node [draw,circle,fill=white,inner sep=2pt,label=above:\tiny{$u_1^+$}] (G) at (0,1) {};
\node [draw,circle,fill=white,inner sep=2pt,label=above:\tiny{$u_2^+$}] (H) at (1,1) {};
\node [draw,circle,fill=white,inner sep=2pt,label=above:\tiny{$u_3^+$}] (I) at (2,1) {};
\node [draw,circle,fill=white,inner sep=2pt,label=above:\tiny{$u_{n-2}^+$}] (J) at (3,1) {};
\node [draw,circle,fill=white,inner sep=2pt,label=above:\tiny{$u_{n-1}^+$}] (K) at (4,1) {};
\node [draw,circle,fill=white,inner sep=2pt,label=above:\tiny{$u_n^+$}] (L) at (5,1) {};

\path [->](A)edge node{}(B);
\path [->](B)edge node{}(C);
\path [->](D)edge node{}(E);
\path [->](E)edge node{}(F);

\path [->](G)edge node{}(H);
\path [->](H)edge node{}(I);
\path [->](J)edge node{}(K);
\path [->](K)edge node{}(L);

\path [->](G)edge node{}(A);
\path [->](H)edge node{}(B);
\path [->](I)edge node{}(C);
\path [->](J)edge node{}(D);
\path [->](K)edge node{}(E);
\path [->](L)edge node{}(F);
\end{tikzpicture}
\caption{Ladder graph $L_n$}
\label{ladder graph}
\end{figure}

Let $S(\G)$ denote the set of all spanning trees of the underlying graph of $\G$ and $s(\G):=\#S(\G)$. The graph obtained from $\G$ by contracting edge $e$ is denoted by $\G/e$. It is known from \cite[p.35]{mr} that $s(L_n)=4s(L_{n-1})-s(L_{n-2})$ for all $n\geq 2$. Since, $s(L_0)=0$ and $s(L_1)=1$, we observe that $s(L_n)=s_n$ for all $n\geq0$.

Let $Q=[q_{u,e}]$ be the incidence matrix of $L_n$ with the orientation given in Fig. \ref{ladder graph} whose rows and columns are indexed by ordered sets
\[V(L_n)=\{u_1^+,\ldots,u_n^+,u_1^-,\ldots,u_n^-\}\]
and
\[E(L_n)=\{f_1,\ldots,f_n,e_1^+\ldots,e_{n-1}^+,e_1^-\ldots,e_{n-1}^-\},\]
respectively. Also, let 
\begin{equation}\label{matrix}
H=\begin{pmatrix} 
B& -B \\
C & D \\
D&C
\end{pmatrix},
\end{equation}
be the $(3n-2)\times 2n$ matrix defined as follows: The rows and columns of $H$ are indexed by ordered sets $E(L_n)$ and $V(L_n)$, respectively. Let $B=[b_{i,j}]$ be the $n\times n$ symmetric matrix given by $b_{i,j}=a_i a_{n-j+1}/2s_n$, and $C=[c_{i,j}]$ and $D=[d_{i,j}]$ be the $(n-1)\times n$ matrices given by
\begin{equation*}
c_{i,j}=\begin{cases}
\delta_{i,j}-\frac{i}{2n}-\frac{a_{n-j+1}s_i}{2s_n},&i\leq j,\\
\frac{1}{2}-\frac{i}{2n}+\frac{a_js_{n-i}}{2s_n},&i>j,
\end{cases}
\end{equation*}
and 
\begin{equation*}
d_{i,j}=\begin{cases}
-\frac{i}{2n}+\frac{a_{n-j+1}s_i}{2s_n},&i\leq j,\\
\frac{1}{2}-\frac{i}{2n}-\frac{a_js_{n-i}}{2s_n},&i>j,
\end{cases}
\end{equation*}
respectively. In the following series of lemmas, we provide the machinery to prove that $H$ is the Moore-Penrose inverse of $Q$.
\begin{lemma}\label{ladder graph: H1=0}
$H\1=0$.
\end{lemma}
\begin{proof}
It is obvious from the definition that $(C+D)\1=0$. Thus 
\[H\1=\begin{bmatrix}
(B-B)\1\\
(C+D)\1\\
(D+C)\1
\end{bmatrix}
=0,\]
as required
\end{proof}
\begin{lemma}\label{ladder graph: QH=I-J/2n}
$QH=I-\frac{1}{2n}J$.
\end{lemma}
\begin{proof}
Let $\delta'_{i,j}:=1-\delta_{i,j}$. For $1\leq i\leq j\leq n$, we have
\begin{align*}
(QH)_{u_i^+,u_j^+}&=b_{i,j}-\delta'_{i,1}c_{i-1,j}+\delta'_{i,n}c_{i,j}\\
&=\frac{a_ia_{n-j+1}}{2s_n}-\delta'_{i,1}\parn{-\frac{i-1}{2n}-\frac{a_{n-j+1}s_{i-1}}{2s_n}}+\delta'_{i,n}\parn{\delta_{i,j}-\frac{i}{2n}-\frac{a_{n-j+1}s_k}{2s_n}}\\
&=\delta_{i,j}-\frac{1}{2n},
\end{align*}
by using equation \eqref{s_n=a_1+...+a_n} when $i=n$. Also, an analogous argument yields 
\[(QH)_{u_j^+,u_i^+}=(QH)_{u_i^-,u_j^-}=(QH)_{u_j^-,u_i^-}=\delta_{i,j}-\frac{1}{2n}.\]
Similarly, by discussing the cases where $i=1$, $i=n$, and $i\neq 1,n$ and either $i\leq j$, $i=j+1$, and $i>j+1$, we obtain
\[(QH)_{u_i^+,u_j^-}=(QH)_{u_i^-,u_j^+}=-b_{i,j}-\delta'_{1,i}d_{i-1,j}+\delta'_{i,n}d_{i,j}=-\frac{1}{2n}\]
for all $1\leq i,j\leq n$. The proof is complete.
\end{proof}
\begin{lemma}\label{ladder graph: HQ symmetric}
$HQ$ is a symmetric matrix.
\end{lemma}
\begin{proof}
Let $1\leq i,j\leq n$. Since $B$ is a symmetric matrix, 
\[(HQ)_{f_i,f_j}=2b_{i,j}=2b_{j,i}=(HQ)_{f_j,f_i}.\] 
From the definition of $H$, we get
\[(HQ)_{e_i^+,f_j}=-(HQ)_{e_i^-,f_j}=c_{i,j}-d_{i,j}=\begin{cases}
\delta_{i,j}-\frac{a_{n-j+1}s_i}{s_n},&i\leq j,\\
\frac{a_js_{n-i}}{s_n},&i>j,
\end{cases}\]
and
\[(HQ)_{f_j,e_i^+}=-(HQ)_{f_j,e_i^-}=b_{j,i}-b_{j,i+1}=\begin{cases}
\frac{a_ia_{n-j+1}-a_{i+1}a_{n-j+1}}{2s_n},&i<j,\\
\frac{a_ja_{n-i+1}-a_ja_{n-i}}{2s_n},&i\geq j.
\end{cases}\]
Utilizing equations \eqref{s_n=a_1+...+a_n}, \eqref{2s_n=a_(n+1)-a_n}, and \eqref{s_n=a_is_(n-i)+a_(n-i+1)s_i} and a simple case-by-case analysis, it yields $(HQ)_{e_i^+,f_j}=(HQ)_{f_j, e_i}$ and $(HQ)_{e_i^-,f_j}=(HQ)_{f_j, e_i^-}$. On the other hand,
\[(HQ)_{e_i^+,e_j^+}=c_{i,j}-c_{i,j+1}=\begin{cases}
\delta_{i,j}+\frac{(a_{n-j}-a_{n-j+1})s_i}{2s_n},&i\leq j,\\
\frac{(a_j-a_{j+1})s_{n-i}}{2s_n},&i>j+1,\\
-\frac{1}{2}+\frac{a_js_{n-i}+a_{n-j}s_i}{2s_n},&i=j+1,
\end{cases}\]

\[(HQ)_{e_i^+,e_j^-}=d_{i,j}-d_{i,j+1}=\begin{cases}
\frac{(a_{n-j+1}-a_{n-j})s_i}{2s_n},&i\leq j,\\
\frac{1}{2}-\frac{a_js_{n-i}+a_{n-j}s_i}{2s_n},&i=j+1,\\
\frac{(a_{j+1}-a_j)s_{n-i}}{2s_n},&i>j+1,
\end{cases}\]

\[(HQ)_{e_j^-,e_i^+}=c_{j,i}-c_{j,i+1}=\begin{cases}
\frac{(a_i-a_{i+1})s_{n-j}}{2s_n},&j>i+1,\\
-\frac{1}{2}+\frac{a_is_{n-j}+a_{n-i}s_j}{2s_n},&j=i+1,\\
\frac{(a_{n-i}-a_{n-i+1})s_j}{2s_n},&i\geq j
\end{cases}\]
and
\[(HQ)_{e_i^-,e_j^-}=d_{i,j}-d_{i,j+1}=\begin{cases}
\frac{(a_{n-j+1}-a_{n-j})s_i}{2s_n},&i\leq j,\\
\frac{1}{2}-\frac{a_js_{n-i}+a_{n-j}s_i}{2s_n},&i=j+1,\\
\frac{(a_{j+1}-a_j)s_{n-i}}{2s_n},&i>j+1,
\end{cases}\]
from which the result follows.
\end{proof}
\begin{theorem}\label{ladder graph: mpinv}
Let $L_n$ be a ladder graph with incident matrix $Q$. Then $Q^+=H$.
\end{theorem}
\begin{proof}
By Lemmas \ref{ladder graph: QH=I-J/2n} and \ref{ladder graph: HQ symmetric}, $QH$ and $HQ$ are symmetric. On the other hand, $QHQ=Q$ for $\1'Q=0$. Since $HQH=H$ by Lemmas \ref{ladder graph: H1=0} and \ref{ladder graph: QH=I-J/2n}, it follows that $H=Q^+$.
\end{proof}

The following lemma will be used frequently in the reminder of this section.
\begin{lemma}[{\cite[Corollary 3.2]{aa-rbb-mfdg}}]\label{resistance}
Let $\G$ be a connected graph and $P$ be a path of length $d$ between vertices $u, v\in V(\G)$. If all edges in $P$ have the same direction from $u$ to $v$ then,
\[r(u,v)=\sum _{e\in E(P)}(q^+_{e,u}-q^+_{e,v}),\]
where $r(u,v)$ is the resistance distance between vertices $u$ and $v$.
\end{lemma}
\begin{corollary}
If $\B=2s(L_n)B$. Then $\B$ has eigenvalue $s(L_n)$ with associated eigenvector $\1$. Also, the diagonal entries of $\B$ satisfy the equations 
\[\B_{i,i}=s(L_n/f_i)\]
for $i=1, \ldots, n$. 
\end{corollary}
\begin{proof}
By \cite[Lemma 2.1]{aa-rbb-ee}
\[\sum_{i=1}^n q^+_{f_i,u_j^+}=1-\frac{n}{2n}=\frac{1}{2},\]
On the other hand, by Theorem \ref{ladder graph: mpinv},
\[\sum_{i=1}^n q^+_{f_i,u_j^+}=\frac{a_i\sum_{i\leq j}a_{n-j+1}+a_{n-i+1}\sum_{i>j}a_j}{2s(L_n)}\]
so that
\[s(L_n)=a_i\sum_{i\leq j}a_{n-j+1}+a_{n-i+1}\sum_{i>j}a_j=\sum_{j=1}^nb_{i,j},\]
that is $\B\1=s(L_n)\1$. Now, from Lemma \ref{resistance} and Theorem \ref{ladder graph: mpinv}, and the definition of the resistance distance between two vertices, we get
\[\frac{s(L_n/f_i)}{s(L_n)}=r(u_i^+,u_i^-)=2q^+_{f_i,u_i^+}=\frac{a_ia_{n-i+1}}{s(L_n)}\]
(see \cite[p. 133]{rbb}). Therefore
\begin{equation}\label{s(L_n/f_i)=a_ia_(n-i+1)}
s(L_n/f_i)=a_ia_{n-i+1}=\B_{i,i},
\end{equation}
as required.
\end{proof}

In \cite{zc} the author uses circus reduction to obtain the resistance distance and Kirchhoff index of ladder graphs. Here we use the Moore-Penrose inverses of Laplacian matrices of ladder graphs to drive new formulas for resistance distance and Kirchhoff index of ladder graphs.
\begin{theorem}\label{ladder graph: resistance distance matrix}
Let $R=[r(u,v)]$ be the resistance distance matrix of the ladder graph $L_n$ of order $2n$. Then
\[r(u_i^\varepsilon,u_j^{\varepsilon'})=\frac{d}{2}-\varepsilon\varepsilon'\frac{a_ia_{n-j+1}}{2s_n}+\alpha(i,j),\]
where $d=d(u_i^+,u_j^+)$, $\varepsilon,\varepsilon'=\pm1$, and $\alpha(i,j)=(s(L_n/f_i)+s(L_n/f_j))/4s_n$. 
\begin{proof}
Let $d=j-i=d(u_i^+,u_j^+)$ and $P:e_i^+,\ldots,e_{j-1}^+$ be the shortest path between $u_i^+$ and $u_j^+$, where $1\leq i<j\leq n$. Then
\[\sum_{k=i}^{j-1}q^+_{e_k^+,u_i^+}=\sum_{k=i}^{j-1}c_{k,i}=c_{i,i}+\sum_{k=i+1}^{j-1}c_{k,i},\]
Thus
\[\sum_{k=i}^{j-1}q^+_{e_k^+,u_i^+}=\parn{1-\frac{i}{2n}-\frac{a_{n-i+1}s_i}{2s_n}}+\sum_{k=i+1}^{j-1}\parn{\frac{1}{2}-\frac{k}{2n}+\frac{a_is_{n-k}}{2s_n}}.\]
By equations \eqref{2s_n=a_(n+1)-a_n} and \eqref{a_(n+1)-a_n=a_(i+1)a_(n-i+1)-a_ia_(n-i)},
\begin{align*}
\sum_{k=i}^{j} q^+_{e_k^+,u_i^+}&=\frac{d+1}{2}-\frac{id}{2n}-\frac{d(d-1)}{4n}+\frac{a_i(a_{n-i}-a_{n-j+1})-a_{n-i+1}(a_{i+1}-a_i)}{4s_n}\\
&=\frac{d+1}{2}-\frac{id}{2n}-\frac{d(d-1)}{4n}+\frac{a_n-a_{n+1}+a_ia_{n-i+1}-a_ia_{n-j+1}}{4s_n}\\
&=\frac{d+1}{2}-\frac{id}{2n}-\frac{d(d-1)}{4n}+\frac{-2s_n+s(L_n/f_i)-a_ia_{n-j+1}}{4s_n}\\
&=\frac{d}{2}-\frac{id}{2n}-\frac{d(d-1)}{4n}-\frac{a_ia_{n-j+1}}{4s_n}+\frac{s(L_n/f_i)}{4s_n}.
\end{align*}
The same argument for the vertex $u_j^+$ yields
\begin{align*}
\sum_{k=i}^{j-1} q^+_{e_k^+,u_j^+}=\sum_{k=i}^{j-1}c_{k,j}&=-\frac{id}{2n}-\frac{d(d-1)}{4n}+\frac{a_{n-j+1}(a_i-a_{d+i})}{4s_n}\\
&=-\frac{id}{2n}-\frac{d(d-1)}{4n}+\frac{a_ia_{n-j+1}}{4s_n}-\frac{s(L_n/f_j)}{4s_n}.
\end{align*}
Therefore,
\begin{align*}
r(u_i^+,u_j^+)=\frac{d}{2}-\frac{a_ia_{n-j+1}}{2s_n}+\frac{s(L_n/f_i)+s(L_n/f_j)}{4s_n}.
\end{align*}
From the block structure of $H$ it follows that $r(u_i^-,u_j^-)=r(u_i^+,u_j^+)$. Finally, from the equality
\begin{align*}
\sum_{k=i}^{j-1}q^+_{e_k^+,u_j^-}=\sum_{k=i}^{j-1}d_{k,j}&=\sum_{k=i}^{j-1}\parn{-\frac{k}{2n}+\frac{a_{n-j+1}s_k}{2s_n}}\\
&=-\frac{id}{2n}-\frac{d(d-1)}{4n}-\frac{a_ia_{n-j+1}}{4s_n}+\frac{s(L_n/f_j)}{4s_n}
\end{align*}
we obtain
\begin{align*}
r(u_j^-,u_i^+)=r(u_i^+,u_j^-)&=\sum_{k=i}^{j-1} q^+_{e_k^+,u_i^+}+b_{j,i}-\sum_{k=i}^{j-1}q^+_{e_k^+,u_j^-}+b_{j,j}\\
&=\frac{d}{2}+\frac{a_ia_{n-j+1}}{2s_n}+\frac{s(L_n/f_i)+s(L_n/f_j)}{4s_n},
\end{align*}
as required.
\end{proof}
\end{theorem}

Utilizing Theorem \ref{ladder graph: resistance distance matrix}, we can compute the Kirchhoff index of ladder graphs.
\begin{corollary}
For any $n\geq1$,
\[Kf(L_n)=\frac{n^2}{3}\parn{n+1+\frac{a_n}{s_n}}.\]
\end{corollary}
\begin{proof}
We have 
\begin{align*}
Kf(L_n)&=\sum_{\{u,v\}\subseteq V(L_n)}r(u,v)\\
&=\frac{4n\sum_{i=1}^ns(L_n/f_i)}{4s_n}+2Kf(P_n)\\
&=\frac{n\sum_{i=1}^na_ia_{n+1-i}}{s_n}+\frac{n(n^2-1)}{3}
\end{align*}
from which the result follows by equation \eqref{a_1a_n+...+a_na_1=(n+1)s_n+2na_n)/6}.
\end{proof}
\section{Circular ladder graph}
In this section, we compute the Moore-Penrose inverse of incidence matrices of circular ladder graphs. Accordingly, we obtain the resistance distance matrix and Kirchhoff index of circular ladders. In what follows, $CL_n$ denotes the circular ladder graph of order $2n$ and $u_1^+,\ldots,u_n^+$ and $u_1^-,\ldots,u_n^-$ stand for the vertices of the internal and external cycles of $CL_n$, respectively, as in Fig. \ref{circular ladder}. Also, $u_{i+kn}^\varepsilon$ denotes the vertex $u_i^\varepsilon$ for all $1\leq i\leq n$, integer $k$, and $\varepsilon=\pm1$.

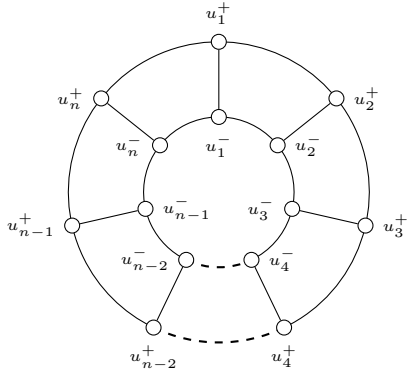
\begin{figure}[h]
\begin{tikzpicture}[>=stealth]
\draw ({cos(90-3*360/7)},{sin(90-3*360/7)}) arc (90-3*360/7:90+3*360/7:1cm);
\draw [dashed, thick] ({cos(90+3*360/7)},{sin(90+3*360/7)}) arc (90+3*360/7:90+4*360/7:1cm);
\draw ({2*cos(90-3*360/7)},{2*sin(90-3*360/7)}) arc (90-3*360/7:90+3*360/7:2cm);
\draw [dashed, thick] ({2*cos(90+3*360/7)},{2*sin(90+3*360/7)}) arc (90+3*360/7:90+4*360/7:2cm);
      
\node [draw,circle,fill=white,inner sep=2pt,label=above:\tiny{$u_1^+$}] (v1) at ({2*cos(90)}, {2*sin(90)}) {};
\node [draw,circle,fill=white,inner sep=2pt,label=right:\tiny{$u_2^+$}] (v2) at ({2*cos(90-360/7)}, {2*sin(90-360/7)}) {};
\node [draw,circle,fill=white,inner sep=2pt,label=right:\tiny{$u_3^+$}] (v3) at ({2*cos(90-2*360/7)}, {2*sin(90-2*360/7)}) {};
\node [draw,circle,fill=white,inner sep=2pt,label=below:\tiny{$u_4^+$}] (v4) at ({2*cos(90-3*360/7)}, {2*sin(90-3*360/7)}) {};
\node [draw,circle,fill=white,inner sep=2pt,label=below:\tiny{$u_{n-2}^+$}] (vn-2) at ({2*cos(90+3*360/7)}, {2*sin(90+3*360/7)}) {};
\node [draw,circle,fill=white,inner sep=2pt,label=left:\tiny{$u_{n-1}^+$}] (vn-1) at ({2*cos(90+2*360/7)}, {2*sin(90+2*360/7)}) {};
\node [draw,circle,fill=white,inner sep=2pt,label=left:\tiny{$u_n^+$}] (vn) at ({2*cos(90+360/7)}, {2*sin(90+360/7)}) {};

\node [draw,circle,fill=white,inner sep=2pt,label=below:\tiny{$u_1^-$}] (u1) at ({cos(90)}, {sin(90)}) {};
\node [draw,circle,fill=white,inner sep=2pt,label=right:\tiny{$u_2^-$}] (u2) at ({cos(90-360/7)}, {sin(90-360/7)}) {};
\node [draw,circle,fill=white,inner sep=2pt,label=left:\tiny{$u_3^-$}] (u3) at ({cos(90-2*360/7)}, {sin(90-2*360/7)}) {};
\node [draw,circle,fill=white,inner sep=2pt,label=right:\tiny{$u_4^-$}] (u4) at ({cos(90-3*360/7)}, {sin(90-3*360/7)}) {};
\node [draw,circle,fill=white,inner sep=2pt,label=left:\tiny{$u_{n-2}^-$}] (un-2) at ({cos(90+3*360/7)}, {sin(90+3*360/7)}) {};
\node [draw,circle,fill=white,inner sep=2pt,label=right:\tiny{$u_{n-1}^-$}] (un-1) at ({cos(90+2*360/7)}, {sin(90+2*360/7)}) {};
\node [draw,circle,fill=white,inner sep=2pt,label=left:\tiny{$u_n^-$}] (un) at ({cos(90+360/7)}, {sin(90+360/7)}) {};

\draw (u1)--(v1);
\draw (u2)--(v2);
\draw (u3)--(v3);
\draw (u4)--(v4);
\draw (un-2)--(vn-2);
\draw (un-1)--(vn-1);
\draw (un)--(vn);
\end{tikzpicture}
\caption{Circular ladder $CL_n$}
\label{circular ladder}
\end{figure}

Let $\G$ be a planar graph embedded on a shaded cylinder as in Fig. \ref{Jordan path}. A Jordan path is a smooth simple path starting from outer border and ending at the inner border such that it meets every face and edge at most once and never paths through vertices (see \cite{mr}). Removing edges crossed by a Jordan path $\J$ results in a graph whose spanning trees depend only on $\J$. On the other hand, every spanning tree of $\G$ determines a Jordan path. Therefore, the set of all spanning trees of $\G$ can be partitioned into sets $S_\J(\G)$ where each set $S_\J(\G)$ corresponds to spanning trees associated to a Jordan path $\J$. Let $s_\J(\G):=\#S_\J(\G)$.

\begin{figure}[h]
\begin{tikzpicture}
\draw [fill=lightgray] (0,0) circle (2.3cm);
\draw [fill=white] (0,0) circle (0.7cm);
\draw (0,0) circle (1cm);
\draw (0,0) circle (1.5cm);
\draw (0,0) circle (2cm);
      
\foreach \x in {0,...,7}
	{
		\draw ({cos(\x*45)}, {sin(\x*45)})--({2*cos(\x*45)}, {2*sin(\x*45)});
		\node [draw, circle, fill=white, inner sep=1pt] () at ({cos(\x*45)}, {sin(\x*45)}) {};
		\node [draw, circle, fill=white, inner sep=1pt] () at ({1.5*cos(\x*45)}, {1.5*sin(\x*45)}) {};
		\node [draw, circle, fill=white, inner sep=1pt] () at ({2*cos(\x*45)}, {2*sin(\x*45)}) {};
	}

\draw [dashed] ({2.5*cos(202.5)}, {2.5*sin(202.5)})--({1.75*cos(202.5)}, {1.75*sin(202.5)});
\draw [dashed] ({1.75*cos(202.5)}, {1.75*sin(202.5)}) arc (202.5:290.5:1.75);
\draw [dashed] ({1.75*cos(292.5)}, {1.75*sin(292.5)})--({1.25*cos(292.5)}, {1.25*sin(292.5)});
\draw [dashed] ({1.25*cos(292.5)}, {1.25*sin(292.5)}) arc (292.5:427.5:1.25);
\draw [dashed] ({1.25*cos(427.5)}, {1.25*sin(427.5)})--({0.5*cos(427.5)}, {0.5*sin(427.5)});
\end{tikzpicture}
\caption{Jordan path}
\label{Jordan path}
\end{figure}
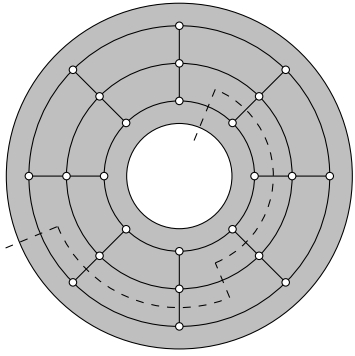
\begin{theorem}\label{circular ladder / f: spanning trees}
Let $CL_n$ be the circular ladder graph of order $2n$. If $f$ is a spoke, then
\[s(CL_n/f)=ns(L_n).\]
\end{theorem}
\begin{proof}
Consider a natural embedding of circular ladder $CL_n$ on a cylinder as in Fig. \ref{Jordan path} and assume without loss of generality that $f=u_1^+u_1^-$. We have two (equivalent up to reflection) classes of Jordan paths:
\begin{itemize}
\item[(1)]Jordan paths $\J_{i,j}^1$ through edges $u_i^+u_{i+1}^+,u_{i+1}^+u_{i+1}^-,\ldots,u_{j-1}^+u_{j-1}^-,u_{j-1}^-u_j^-$ for $1\leq i<j\leq n+1$;
\item[(2)]Jordan paths $\J_{i,j}^2$ through edges $u_j^+u_{j-1}^+,u_{j-1}^+u_{j-1}^-,\ldots,u_{i+1}^+u_{i+1}^-,u_i^-u_{i+1}^-$ for $1\leq i<j\leq n+1$.
\end{itemize}
Note that the Jordan paths in parts (1) and (2) coincide when $j=i+1$. Also $s_{\J_{i,j}^1}(CL_n/f)=s_{\J_{i,j}^2}(CL_n/f)$. Utilizing the above Jordan paths in conjunction with equation \eqref{s(L_n/f_i)=a_ia_(n-i+1)}, yields
\begin{align*}
s(CL_n/f)=&2\sum_{1\leq i<j\leq n+1}s_{\J_{i,j}^1}(CL_n/f)-\sum_{i=1}^ns_{\J_{i,i+1}^1}(CL_n/f)\\
=&2\sum_{1\leq i<j\leq n+1}s(L_{n+1+i-j}/f_i)-\sum_{i=1}^ns(L_n/f_i)\\
=&2\sum_{1\leq i<j\leq n+1}a_ia_{n+2-j}-\sum_{i=1}^na_ia_{n+1-i}\\
=&2\sum_{2\leq j\leq n+1}s_{j-1}a_{n+2-j}-\sum_{i=1}^na_ia_{n+1-i}\\
=&2\sum_{1\leq i\leq n}a_is_{n+1-i}-\sum_{i=1}^na_ia_{n+1-i},
\end{align*}
where $L_s/f_t$ denotes the contraction of the $t$'s spoke of $L_s$ from right or left. Let $x_n:=\sum_{i=1}^na_ia_{n+1-i}$ and $y_n:=\sum_{i=1}^na_is_{n+1-i}$ for all $n\geq1$. Clearly, $y_n-y_{n-1}=x_n$. Also, from equation \eqref{s_n=a_is_(n-i)+a_(n-i+1)s_i}, it follows that $y_n+y_{n-1}=ns_n$. Thus $2y_n=ns_n+x_n$, which implies that 
\[s(CL_n/f)=2y_n-x_n=ns_n,\]
as required.
\end{proof}
\begin{theorem}\label{circular ladder: spanning trees}
Let $CL_n$ be the circular ladder graph of order $2n$. Then
\[s(CL_n)=\frac{1}{2}n(a_{n+1}+a_n-2).\]
\end{theorem}
\begin{proof}
Analogous to the proof of Theorem \ref{circular ladder / f: spanning trees}, we consider two families of (equivalent up to reflection) Jordan paths given by $\J_{i,j}^1$ and $\J_{i,j}^2$ as in Theorem \ref{circular ladder / f: spanning trees} where $i$ is behind $j$ regarding clockwise rotations. From the rotational symmetry of $CL_n$ and the equality of $\J_{i,j}^1$ and $\J_{i,j}^2$ for consecutive values of $i$ and $j$, it follows that
\begin{align*}
s(CL_n)&=2n\sum_{i=1}^ns(L_i)-ns(L_n)=\frac{1}{2}n(a_{n+1}-a_n-2)
\end{align*}
by equation \eqref{2s_n=a_(n+1)-a_n}.
\end{proof}

In what follows, $e_i^+$, $e_i^-$, and $f_i$ stand for directed edges $(u_i^+,u_{i+1}^+)$, $(u_i^-,u_{i+1}^-)$, and $(u_i^+,u_i^-)$ for $i=1,\ldots,n$. Accordingly, we consider the orientation of $CL_n$ whose directed edges are $e_i^+$, $e_i^-$, and $f_i$ for $i=1,\ldots,n$.
\begin{theorem}\label{circular ladder: mpinv}
Let $CL_n$ be the oriented circular ladder graph of order $2n$ with incidence matrix $Q$. The Moore-Penrose inverse $Q^+=[q_{e,u}^+]$ of $Q$ is given by
\begin{align*}
q^+_{e_{i+t}^{\varepsilon'},u_i^\varepsilon}&=-\varepsilon\varepsilon'\frac{2a_{t+1}-a_t}{4}\cdot\frac{a_{n+1}-a_n}{a_{n+1}+a_n-2}+\varepsilon\varepsilon'\frac{a_{t+1}+\varepsilon\varepsilon'1}{4}-\frac{2t+1}{4n},\\
\varepsilon q^+_{f_{i+t},u_i^\varepsilon}&=\frac{a_{t+1}+a_t}{4}\cdot\frac{a_{n+1}-a_n}{a_{n+1}+a_n-2}-\frac{a_{t+1}-a_t}{4}
\end{align*}
for all $1\leq i\leq n$, $t=0,\ldots,n-1$, and $\varepsilon,\varepsilon'=\pm1$.
\end{theorem}
\begin{proof}
Let $d=d(f_i,u_j)$. From the symmetry of graph, it follows that $q^+_{f_i,u_i^-}=-q^+_{f_i,u_i^+}$. Hence
\[2q^+_{f_i,u_i^+}=r(u_i^+,u_i^-)=\frac{s(CL_n/f_i)}{s(CL_n)}=\frac{s_n}{a_{n+1}+a_n-2}\]
by Lemma \ref{resistance} and Theorems \ref{circular ladder / f: spanning trees} and \ref{circular ladder: spanning trees}. 

Consider the edge cut-set $\{e_j^+,f_{j+1},e_{j+1}^+\}$, where $j=i+t$ with $0\leq t<n$. From \cite[Lemma 2.1]{aa-rbb-ee}, it follows that
\begin{equation}\label{CL_n: q^+_(e_j^+,u_i^+)}
q^+_{e_{j+1}^+,u_i^+}=\delta_{i+1,j}-\frac{1}{2n}-q^+_{f_{j+1},u_i^+}+q^+_{e_j^+,u_i^+}.
\end{equation}
Analogously, by considering the edge cut-set $\{e_j^-,f_{j+1},e_{j+1}^-\}$, we obtain
\begin{equation}\label{CL_n: q^+_(e_j^-,u_i^+)}
q^+_{e_{j+1}^-,u_i^+}=-\frac{1}{2n}+q^+_{f_{j+1},u_i^+}+q^+_{e_j^-,u_i^+}.
\end{equation}
for all $j\geq i$. If we replace $j$ by $i-1$, and use the equalities 
\[q^+_{e_{i-1}^+,u_i^+}=-q^+_{e_i^+,u_i^+}\quad\text{and}\quad q^+_{e_{i-1}^-,u_i^+}=-q^+_{e_i^-,u_i^+}\]
obtained from the symmetry of the graph, then we get the initial values
\[q^+_{e_i^+,u_i^+}=\frac{1}{2}\parn{1-\frac{1}{2n}-q^+_{f_i,u_i^+}}\quad\text{and}\quad q^+_{e_i^-,u_i^+}=\frac{1}{2}\parn{\frac{-1}{2n}+q^+_{f_i,u_i^+}}.\]
On the other hand, the incidence vector of the cycle induced by $\{u_j^+,u_{j+1}^+,u_j^-,u_{j+1}^-\}$ lies in the left null space of $Q^+$, from which we obtain
\[q^+_{f_{j+1},u_i^+}-q^+_{e_j^-,u_i^+}-q^+_{f_j,u_i^+}+q^+_{e_j^+,u_i^+}=0\]
so that
\begin{equation}\label{CL_n: q^+_(f_(j+1),u_i^+)}
q^+_{f_{j+1},u_i^+}=q^+_{e_j^-,u_i^+}+q^+_{f_j,u_i^+}-q^+_{e_j^+,u_i^+}
\end{equation}
for all $j\geq i$. Let 
\[C:=\begin{pmatrix}
2&-1&-1\\
-1&2&1\\
-1&1&1
\end{pmatrix}
\quad\text{and}\quad
D:=-\frac{1}{2n}\begin{pmatrix}
1\\1\\0
\end{pmatrix}.\]
If $X_j$ denotes the column matrix $(q^+_{e_j^+,u_i^+}\ q^+_{e_j^-,u_i^+}\ q^+_{f_j,u_i^+})^T$ for $i\leq j<i+n$, then from the equalities \eqref{CL_n: q^+_(e_j^-,u_i^+)} and \eqref{CL_n: q^+_(f_(j+1),u_i^+)}, it follows that
\[X_{j+1}=CX_j+D\]
for all $i\leq j<i+n$. Using induction on $t\geq0$, one can prove that
\[C^t=\begin{pmatrix}
\frac{a_{t+1}+1}{2}&-\frac{a_{t+1}-1}{2}&-s_t\\
-\frac{a_{t+1}-1}{2}&\frac{a_{t+1}+1}{2}&s_t\\
-s_t&s_t&a_t
\end{pmatrix}\]
and
\[I+C+\cdots+C^{t-1}=\begin{pmatrix}
\frac{s_t+t}{2}&-\frac{s_t-t}{2}&-\frac{a_t-1}{2}\\
-\frac{s_t-t}{2}&\frac{s_t+t}{2}&\frac{a_t-1}{2}\\
-\frac{a_t-1}{2}&\frac{a_t-1}{2}&s_{t-1}+1
\end{pmatrix}.\]
Now, from the equality
\[X_{i+t}=C^tX_i+(I+C+\cdots+C^{t-1})D,\]
it follows that
\begin{align*}
q^+_{e_{i+t}^+,u_i^+}&=-\frac{2a_{t+1}-a_t}{4}\cdot\frac{a_{n+1}-a_n}{a_{n+1}+a_n-2}+\frac{a_{t+1}+1}{4}-\frac{2t+1}{4n},\\
q^+_{e_{i+t}^-,u_i^+}&=\frac{2a_{t+1}-a_t}{4}\cdot\frac{a_{n+1}-a_n}{a_{n+1}+a_n-2}-\frac{a_{t+1}-1}{4}-\frac{2t+1}{4n},\\
q^+_{f_{i+t},u_i^+}&=\frac{a_{t+1}+a_t}{4}\cdot\frac{a_{n+1}-a_n}{a_{n+1}+a_n-2}-\frac{a_{t+1}-a_t}{4}
\end{align*}
for all $t=0,\ldots,n-1$. Finally, the result follows by applying the same argument or using the symmetry of graph for vertices $u_1^-,\ldots,u_n^-$.
\end{proof}
\begin{corollary}
If $n=2k+1$, then
\[s(CL_n)=n s(L_n)\parn{\frac{s(L_{k+1})+s(L_{k})}{s(L_{k+1})-s(L_{k})}}\]
and if $n=2k$, then
\[s(CL_n)=ns(L_n)\parn{\frac{s(L_{k+1})+2s(L_k)+s(L_{k-1})}{s(L_{k+1})-s(L_{k-1})}}.\]
\end{corollary}
\begin{proof}
We know that $s(CL_n/f_i)/s(CL_n)=r(u_i^+,u_i^-)=2q^+_{f_i,u_i^+}$. Thus
\[s(CL_n)=ns(L_n)\cdot\frac{a_{n+1}+a_n-2}{a_{n+1}-a_n}.\]
Using Binet's formulas, one can easily show that 
 \[\frac{a_{n+1}+a_n-2}{a_{n+1}-a_n}=\frac{s_{k+1}+s_k}{s_{k+1}-s_k}\quad\text{or}\quad\frac{s_{k+1}+2s_k+s_{k-1}}{s_{k+1}-s_k}\]
according to $n=2k+1$ or $n=2k$, respectively. The result follows.
\end{proof}
\begin{corollary}\label{circular ladder: resistance}
Let $1\leq i\leq n$, $0\leq t<n$, and $\varepsilon,\varepsilon'=\pm1$. Then
\[r(u_i^\varepsilon,u_{i+t}^{\varepsilon'})=-\varepsilon\varepsilon'\frac{a_{t+1}+a_t-\varepsilon\varepsilon'2}{4}\cdot\frac{a_{n+1}-a_n}{a_{n+1}+a_n-2}+\varepsilon\varepsilon'\frac{a_{t+1}-a_t}{4}+\frac{t}{2}-\frac{t^2}{2n}.\]
\end{corollary}
\begin{proof}
First observe that $r(u_i^-,u_{i+t}^-)=r(u_i^+,u_{i+t}^+)=r(u_0^+,u_t^+)$ by Theorem \ref{circular ladder: mpinv} and Lemma \ref{resistance}. We have
\[\sum_{s=0}^{t-1}q^+_{e_s^+,u_0^+}=-\frac{a_{t+1}+a_t-2}{8}\cdot\frac{a_{n+1}-a_n}{a_{n+1}+a_n-2}+\frac{a_{t+1}-a_t}{8}+\frac{t}{4}-\frac{t^2}{4n}\]
and 
\[\sum_{s=0}^{t-1}q^+_{e_s^+,u_0^-}=\frac{a_{t+1}+a_t-2}{8}\cdot\frac{a_{n+1}-a_n}{a_{n+1}+a_n-2}-\frac{a_{t+1}-a_t}{8}+\frac{t}{4}-\frac{t^2}{4n}.\]
Also,
\[\sum_{s=0}^{t-1}q^+_{e_s^+,u_t^+}=\sum_{s=0}^{t-1}q^+_{e_{n+s-t}^+,u_0^+}=-\sum_{s=0}^{t-1}q^+_{e_s^+,u_0^+}\]
and
\[\sum_{s=0}^{t-1}q^+_{e_s^+,u_t^-}=\sum_{s=0}^{t-1}q^+_{e_{n+s-t}^+,u_0^-}=-\sum_{s=0}^{t-1}q^+_{e_s^+,u_0^-}\]
by symmetry of the graph. Let $1\leq i\leq n$ and $0\leq t<n$. Then, by Lemma \ref{resistance},
\begin{align*}
r(u_i^+,u_{i+t}^+)&=r(u_0^+,u_t^+)=2\sum_{s=0}^{t-1}q^+_{e_s^+,u_0^+}\\
&=-\frac{a_{t+1}+a_t-2}{4}\cdot\frac{a_{n+1}-a_n}{a_{n+1}+a_n-2}+\frac{a_{t+1}-a_t}{4}+\frac{t}{2}-\frac{t^2}{2n}
\end{align*}
and
\begin{align*}
r(u_i^+,u_{i+t}^-)&=r(u_0^+,u_t^-)=\sum_{s=0}^{t-1}q^+_{e_s^+,u_0^+}-\sum_{s=0}^{t-1}q^+_{e_s^+,u_t^-}+q^+(f_t,u_0^+)-q^+(f_t,u_t^-)\\
&=\sum_{s=0}^{t-1}q^+_{e_s^+,u_0^+}+\sum_{s=0}^{t-1}q^+_{e_s^+,u_0^-}+q^+(f_t,u_0^+)+q^+(f_0,u_0^+)\\
&=\frac{a_{t+1}+a_t+2}{4}\frac{a_{n+1}-a_n}{a_{n+1}+a_n-2}-\frac{a_{t+1}-a_t}{4}+\frac{t}{2}-\frac{t^2}{2n},
\end{align*}
as required.
\end{proof}
\begin{corollary}\label{circular ladder: Kirchhoff index}
For every $n\geq3$, we have
\[Kf(CL_n)=n^2\cdot\frac{a_{n+1}-a_n}{a_{n+1}+a_n-2}+\frac{1}{6}n(n^2-1)\]
\end{corollary}

It is known from \cite[Theorem F]{djk-mr} that $Kf(\G)=n\tr(L^+(\G))$ for every graph $\G$ of order $n$, where $L(\G)$ denotes the Laplacian matrix of $\G$. Also, from the definition, we know that 
\[r(u,v)=L^+_{u,u}+L^+_{v,v}-L^+_{u,v}-L^+_{v,u}\]
for all $u,v\in V(\G)$, where $L^+$ denotes the Moore-Penrose inverse of $L=L(\G)$ and $r(u,v)$ is the resistance distance between vertices $u$ and $v$ of $\G$. Utilizing the fact that circular ladders are vertex transitive graphs, we obtain the following result immediately.
\begin{corollary}\label{circular ladder: mpinv of laplacian matrix}
Let $L:=L(CL_n)$ be the Laplacian matrix of $CL_n$ and $L^+=[l_{uv}^+]$ be its Moore-Penrose inverse. Then
\[l_{u_i^\varepsilon,u_{i+t}^{\varepsilon'}}=\varepsilon\varepsilon'\frac{a_{t+1}+a_t}{8}\cdot\frac{a_{n+1}-a_n}{a_{n+1}+a_n-2}-\varepsilon\varepsilon'\frac{a_{t+1}-a_t}{8}-\frac{t}{4}+\frac{t^2}{4n}+\frac{1}{24}\parn{n-\frac{1}{n}}\]
for all $1\leq i\leq n$, $0\leq t<n$, and $\varepsilon,\varepsilon'=\pm1$.
\end{corollary}
\section{M\"obius ladder graph}
Following the same techniques used in the previous section, we shall compute the Moore-Penrose inverse of incidence matrices of M\"{o}bius ladder graphs. Likewise, we obtain the resistance distance matrix and Kirchhoff index of M\"{o}bius ladders. In what follows, $M_n$ denotes the M\"{o}bius ladder graph of order $2n$ and $u_1,\ldots,u_n$ and $v_1,\ldots,v_n$ stand for the vertices of the external and internal cycles of $M_n$, respectively, as in Fig. \ref{Mobius ladder}(a). Also, we may set $u_{i+n}:=v_i$ for all $i=1,\ldots,n$, and that $u_{i+2nk}:=u_i$ for all $1\leq i\leq 2n$ and integers $k$. In what follows, $f_i:=u_iv_i$ ($1\leq i\leq n$) is the $i$-th spoke and $e_i=u_iu_{i+1}$ is the $i$-th edge of the $2n$-cycle $u_1,\ldots,u_{2n}$ ($1\leq i\leq 2n$), see the second drawing in Fig. \ref{Mobius ladder}(b).

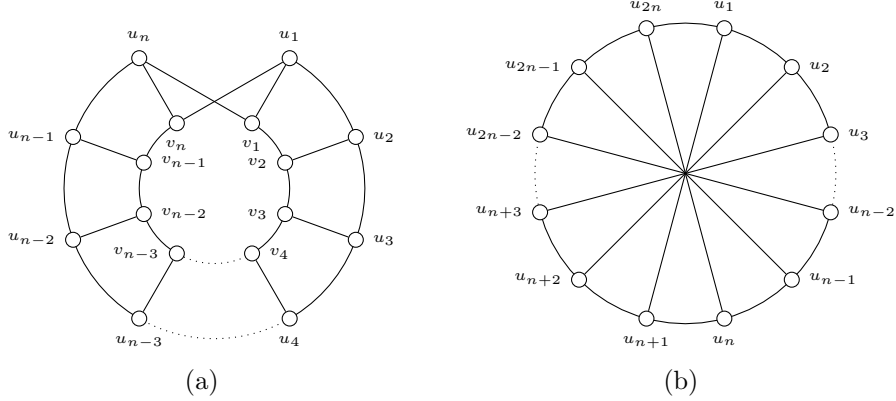
\begin{figure}[h]
\begin{tabular}{ccc}
\begin{tikzpicture}[>=stealth]
\draw [domain=60:-60] plot ({2*cos(\x)}, {2*sin(\x)});
\draw [domain=120:240] plot ({2*cos(\x)}, {2*sin(\x)});
\draw [domain=60:-60] plot ({cos(\x)}, {sin(\x)});
\draw [domain=120:240] plot ({cos(\x)}, {sin(\x)});
\draw [dotted,domain=-60:-120] plot ({2*cos(\x)}, {2*sin(\x)});
\draw [dotted,domain=-60:-120] plot ({cos(\x)}, {sin(\x)});
      
\node [draw,circle,fill=white,inner sep=2pt,label=above:\tiny{$u_1$}] (A) at ({2*cos(60)}, {2*sin(60)}) {};
\node [draw,circle,fill=white,inner sep=2pt,label=right:\tiny{$u_2$}] (B) at ({2*cos(20)}, {2*sin(20)}) {};
\node [draw,circle,fill=white,inner sep=2pt,label=right:\tiny{$u_3$}] (C) at ({2*cos(-20)}, {2*sin(-20)}) {};
\node [draw,circle,fill=white,inner sep=2pt,label=below:\tiny{$u_4$}] (D) at ({2*cos(-60)}, {2*sin(-60)}) {};
\node [draw,circle,fill=white,inner sep=2pt,label=below:\tiny{$u_{n-3}$}] (E) at ({2*cos(-120)}, {2*sin(-120)}) {};
\node [draw,circle,fill=white,inner sep=2pt,label=left:\tiny{$u_{n-2}$}] (F) at ({2*cos(-160)}, {2*sin(-160)}) {};
\node [draw,circle,fill=white,inner sep=2pt,label=left:\tiny{$u_{n-1}$}] (G) at ({2*cos(-200)}, {2*sin(-200)}) {};
\node [draw,circle,fill=white,inner sep=2pt,label=above:\tiny{$u_n$}] (H) at ({2*cos(-240)}, {2*sin(-240)}) {};

\node [draw,circle,fill=white,inner sep=2pt,label=below:\tiny{$v_1$}] (I) at ({cos(60)}, {sin(60)}) {};
\node [draw,circle,fill=white,inner sep=2pt,label=left:\tiny{$v_2$}] (J) at ({cos(20)}, {sin(20)}) {};
\node [draw,circle,fill=white,inner sep=2pt,label=left:\tiny{$v_3$}] (K) at ({cos(-20)}, {sin(-20)}) {};
\node [draw,circle,fill=white,inner sep=2pt,label=right:\tiny{$v_4$}] (L) at ({cos(-60)}, {sin(-60)}) {};
\node [draw,circle,fill=white,inner sep=2pt,label=left:\tiny{$v_{n-3}$}] (M) at ({cos(-120)}, {sin(-120)}) {};
\node [draw,circle,fill=white,inner sep=2pt,label=right:\tiny{$v_{n-2}$}] (N) at ({cos(-160)}, {sin(-160)}) {};
\node [draw,circle,fill=white,inner sep=2pt,label=right:\tiny{$v_{n-1}$}] (O) at ({cos(-200)}, {sin(-200)}) {};
\node [draw,circle,fill=white,inner sep=2pt,label=below:\tiny{$v_n$}] (P) at ({cos(-240)}, {sin(-240)}) {};

\draw (A)--(I);
\draw (B)--(J);
\draw (C)--(K);
\draw (D)--(L);
\draw (E)--(M);
\draw (F)--(N);
\draw (G)--(O);
\draw (H)--(P);
\draw (A)--(P);
\draw (I)--(H);
\end{tikzpicture}
&&
\begin{tikzpicture}
\draw [domain=15:165] plot ({2*cos(\x)}, {2*sin(\x)});
\draw [domain=195:345] plot ({2*cos(\x)}, {2*sin(\x)});
\draw [dotted, domain=15:-15] plot ({2*cos(\x)}, {2*sin(\x)});
\draw [dotted, domain=165:195] plot ({2*cos(\x)}, {2*sin(\x)});

\node [draw,circle,fill=white,inner sep=2pt,label=above:\tiny{$u_1$}] (A) at ({2*cos(75)}, {2*sin(75)}) {};
\node [draw,circle,fill=white,inner sep=2pt,label=right:\tiny{$u_2$}] (B) at ({2*cos(45)}, {2*sin(45)}) {};
\node [draw,circle,fill=white,inner sep=2pt,label=right:\tiny{$u_3$}] (C) at ({2*cos(15)}, {2*sin(15)}) {};
\node [draw,circle,fill=white,inner sep=2pt,label=right:\tiny{$u_{n-2}$}] (D) at ({2*cos(345)}, {2*sin(345)}) {};
\node [draw,circle,fill=white,inner sep=2pt,label=right:\tiny{$u_{n-1}$}] (E) at ({2*cos(315)}, {2*sin(315)}) {};
\node [draw,circle,fill=white,inner sep=2pt,label=below:\tiny{$u_n$}] (F) at ({2*cos(285)}, {2*sin(285)}) {};
\node [draw,circle,fill=white,inner sep=2pt,label=below:\tiny{$u_{n+1}$}] (G) at ({2*cos(255)}, {2*sin(255)}) {};
\node [draw,circle,fill=white,inner sep=2pt,label=left:\tiny{$u_{n+2}$}] (H) at ({2*cos(225)}, {2*sin(225)}) {};
\node [draw,circle,fill=white,inner sep=2pt,label=left:\tiny{$u_{n+3}$}] (I) at ({2*cos(195)}, {2*sin(195)}) {};
\node [draw,circle,fill=white,inner sep=2pt,label=left:\tiny{$u_{2n-2}$}] (J) at ({2*cos(165)}, {2*sin(165)}) {};
\node [draw,circle,fill=white,inner sep=2pt,label=left:\tiny{$u_{2n-1}$}] (K) at ({2*cos(135)}, {2*sin(135)}) {};
\node [draw,circle,fill=white,inner sep=2pt,label=above:\tiny{$u_{2n}$}] (L) at ({2*cos(105)}, {2*sin(105)}) {};

\draw (A)--(G);
\draw (B)--(H);
\draw (C)--(I);
\draw (D)--(J);
\draw (E)--(K);
\draw (F)--(L);
\end{tikzpicture}\\
(a)&&(b)
\end{tabular}
\caption{M\"obius ladder graph $M_n$}
\label{Mobius ladder}
\end{figure}
\begin{theorem}\label{Mobius ladder: mpinv}
Let $M_n$ be the M\"{o}bius ladder graph of order $2n$ and $Q$ be its incidence matrix. Then the Moore-Penrose inverse of $Q=[q^+_{e,v}]$ is given by
\begin{align*}
q^+_{e_{i+t},u_i}&=-\frac{2a_{t+1}-a_t}{4}\cdot\frac{a_{n+1}-a_n}{a_{n+1}+a_n+2}+\frac{a_{t+1}+1}{4}-\frac{2t+1}{4n},\\
q^+_{e_{n+i+t},u_i}&=\frac{2a_{t+1}-a_t}{4}\cdot\frac{a_{n+1}-a_n}{a_{n+1}+a_n+2}-\frac{a_{t+1}-1}{4}-\frac{2t+1}{4n},\\
\varepsilon_{i+t} q^+_{f_{i+t},u_i}&=\frac{a_{t+1}+a_t}{4}\cdot\frac{a_{n+1}-a_n}{a_{n+1}+a_n+2}-\frac{a_{t+1}-a_t}{4}
\end{align*}
for all $t=0,\ldots,n-1$.
\end{theorem}
\begin{proof}
First observe that $s(M_n/f_i)=s(CL_n/f_i)$ and that by \cite[p. 41]{mr}
\[s(M_n)=s(CL_n)+2n=\frac{n}{2}(a_{n+1}+a_n+2).\]
Therefore, by Lemma \ref{resistance} and Theorem \ref{circular ladder / f: spanning trees}, 
\[q^+_{f_i,u_i}=-q^+_{f_i,v_i}=\frac{r(u_i,v_i)}{2}=\frac{s(M_n/f_i)}{2s(M_n)}=\frac{s_n}{a_{n+1}+a_n+2}.\]
Consider the edge cut-set $\{e_j,f_{j+1},e_{j+1}\}$. From \cite[Lemma 2.1]{aa-rbb-ee}, it follows that
\begin{equation}\label{M_n: q^+_(e_j^+,u_i^+)}
q^+_{e_{j+1},u_i}=\delta_{i,j+1}-\frac{1}{2n}+q^+_{e_j,u_i}-\varepsilon_{j+1} q^+_{f_{j+1},u_i},
\end{equation}
where $\varepsilon_j=(-1)^{[(j-1)/n]}$. Since $q^+_{e_i,u_i}=-q^+_{e_{i-1},u_i}$ by the symmetry of the graph, equation \eqref{M_n: q^+_(e_j^+,u_i^+)} yields the initial values
\[q^+_{e_i,u_i}=\frac{1}{2}\parn{1-\frac{1}{2n}-\varepsilon_i q^+_{f_i,u_i}}\]
for all $i=1,\ldots,2n$. Likewise, from $f_{n+i}=f_i$ and $q^+_{e_{n+i},u_i}=-q^+_{e_{n+i-1},u_i}$ it follows  that
\[q^+_{e_{n+i},u_i}=\frac{1}{2}\parn{-\frac{1}{2n}+\varepsilon_i q^+_{f_i,u_i}}\]
for all $i=1,\ldots,2n$. Notice that the values of $q^+_{e_{i+t},u_i}$ are independent of $i$ by the symmetry of the graph, so we may restrict ourselves to $i$ with $1\leq i\leq n$ and assume $\varepsilon_i=1$.

On the other hand, the incidence vector of the cycle $u_j,u_{j+1},u_{n+j},u_{n+j+1}$ lies in the left null space of $Q^+$, from which it follows that
\[q^+_{e_j,u_i}+\varepsilon_{j+1}q^+_{f_{j+1},u_i}-q^+_{e_{n+j},u_i}-\varepsilon_jq^+_{f_j,u_i}=0\]
or
\begin{equation}\label{M_n: q^+_(f_(j+1),u_i^+)}
\varepsilon_{j+1}q^+_{f_{j+1},u_i}=q^+_{e_{n+j},u_i}-q^+_{e_j,u_i}+\varepsilon_jq^+_{f_j,u_i}.
\end{equation}
Let 
\[C:=\begin{pmatrix}
2&-1&-1&0\\
-1&2&0&-1\\
-1&1&1&0\\
1&-1&0&1\\
\end{pmatrix}
\quad\text{and}\quad
D:=-\frac{1}{2n}\begin{pmatrix}
1\\1\\0\\0
\end{pmatrix}.\]
If $X_j$ denotes the column matrix $(q^+_{e_j,u_i}\ q^+_{e_{n+j},u_i}\ \varepsilon_j q^+_{f_j,u_i}\ \varepsilon_{n+j} q^+_{f_{n+j},u_i})^T$ for $i\leq j<i+n$, then from the equalities \eqref{M_n: q^+_(e_j^+,u_i^+)} and \eqref{M_n: q^+_(f_(j+1),u_i^+)}, it follows that
\[X_{j+1}=CX_j+D\]
for all $i\leq j<n+i-1$. Using induction on $t\geq0$, one can prove that
\[C^t=\begin{pmatrix}
\frac{a_{t+1}+1}{2}&-\frac{a_{t+1}-1}{2}&-\frac{s_t+t}{2}&\frac{s_t-t}{2}\\
-\frac{a_{t+1}-1}{2}&\frac{a_{t+1}+1}{2}&\frac{s_t-t}{2}&-\frac{s_t+t}{2}\\
-s_t&s_t&\frac{a_t+1}{2}&-\frac{a_t-1}{2}\\
s_t&-s_t&-\frac{a_t-1}{2}&\frac{a_t+1}{2}\\
\end{pmatrix}\]
and
\[I+C+\cdots+C^{t-1}=\begin{pmatrix}
\frac{s_t+t}{2}&-\frac{s_t-t}{2}&-\frac{a_t+t^2-t-1}{4}&\frac{a_t-t^2+t-1}{4}\\
-\frac{s_t-t}{2}&\frac{s_t+t}{2}&\frac{a_t-t^2+t-1}{4}&-\frac{a_t+t^2-t-1}{4}\\
-\frac{a_t-1}{2}&\frac{a_t-1}{2}&\frac{s_{t-1}+t+1}{2}&-\frac{s_{t-1}-(t-1)}{2}\\
\frac{a_t-1}{2}&-\frac{a_t-1}{2}&-\frac{s_{t-1}-(t-1)}{2}&\frac{s_{t-1}+t+1}{2}\\
\end{pmatrix}.\]
Now, from the equality
\[X_{i+t}=C^tX_i+(I+C+\cdots+C^{t-1})D,\]
it follows that
\begin{align*}
q^+_{e_{i+t},u_i}&=-\frac{2a_{t+1}-a_t}{4}\cdot\frac{a_{n+1}-a_n}{a_{n+1}+a_n+2}+\frac{a_{t+1}+1}{4}-\frac{2t+1}{4n},\\
q^+_{e_{n+i+t},u_i}&=\frac{2a_{t+1}-a_t}{4}\cdot\frac{a_{n+1}-a_n}{a_{n+1}+a_n+2}-\frac{a_{t+1}-1}{4}-\frac{2t+1}{4n},\\
\varepsilon_{i+t} q^+_{f_{i+t},u_i}&=\frac{a_{t+1}+a_t}{4}\cdot\frac{a_{n+1}-a_n}{a_{n+1}+a_n+2}-\frac{a_{t+1}-a_t}{4}
\end{align*}
for all $t=0,\ldots,n-1$, as required.
\end{proof}
\begin{corollary}\label{Mobius ladder: resistance}
Let $1\leq i\leq 2n$ and $0\leq t\leq n$. Then
\[r(u_i,u_{i+t})=-\frac{a_{t+1}+a_t-2}{4}\cdot\frac{a_{n+1}-a_n}{a_{n+1}+a_n+2}+\frac{a_{t+1}-a_t}{4}+\frac{t}{2}-\frac{t^2}{2n}.\]
\end{corollary}
\begin{proof}
The result follows from Lemma \ref{resistance} and Theorem \ref{Mobius ladder: mpinv} in conjunction with the fact that
\[\sum_{j=0}^{t-1}q^+_{e_{i+j},u_{i+t}}=-\sum_{j=0}^{t-1}q^+_{e_{i+j},u_i}\]
by the symmetry of the graph.
\end{proof}

Theorem \ref{Mobius ladder: resistance} in conjunction with equation \eqref{(a_(n+1)+a_n)^2=3(a_(n+1)-a_n)^2+4} yields
\begin{corollary}\label{Mobius ladder: Kirchhoff index}
For every $n\geq3$, we have
\[Kf(CL_n)=n^2\cdot\frac{a_{n+1}-a_n}{a_{n+1}+a_n+2}+\frac{1}{6}n(n^2-1).\]
\end{corollary}

Finally, the Moore-Penrose inverse of the Laplacian matrix of $M_n$ can be obtained analogous to Corollary \ref{circular ladder: mpinv of laplacian matrix}.
\begin{corollary}\label{Mobius ladder: mpinv of laplacian matrix}
Let $L:=L(M_n)$ be the Laplacian matrix of $M_n$ and $L^+=[l_{uv}^+]$ be its Moore-Penrose inverse. Then
\[l^+_{u_i,u_{i+t}}=\frac{a_{t+1}+a_t}{8}\cdot\frac{a_{n+1}-a_n}{a_{n+1}+a_n+2}-\frac{a_{t+1}-a_t}{8}-\frac{t}{4}+\frac{t^2}{4n}+\frac{1}{24}\parn{n-\frac{1}{n}}\]
for all $1\leq i\leq 2n$ and $0\leq t<n$.
\end{corollary}

\end{document}